\documentclass[12pt]{amsart}
\usepackage{amssymb,bbold}
\usepackage{euscript}
\usepackage[hypertex]{hyperref}
\usepackage[all]{xy}
\catcode`\@=11
\def\@evenfoot{\rule{0pt}{20pt}[\today] \hfill}
\def\@oddfoot{\rule{0pt}{20pt}\hfill [\today]}
\catcode`\@=13
\def\today{July 23}

\newtheorem{theorem}{Theorem}[section]

\newtheorem{proposition}[theorem]{Proposition}
\newtheorem{lemma}[theorem]{Lemma}

\newtheorem{corollary}[theorem]{Corollary}
\theoremstyle{definition}
\newtheorem{definition}[theorem]{Definition}
\newtheorem{example}[theorem]{Example}
\newtheorem{remark}[theorem]{Remark}

\def\Lie{{\mathcal{L}{\it ie}}}\def\hh{h}
\def\calR{{\mathcal R}}
\def\Span{{\rm Span}}\def\Der{{\rm Der}}\def\Ext{{\hbox{\large $\land$}}}
\def\mfI{{\mathfrak I}}
\def\Def{{\mathfrak{Def}}}\def\Diag{{\tt D}}\def\A{{\mathbf A}}
\def\B{{\mathbf B}}\def\Com{{\EuScript C}{\it om}}
\def\twe{{{\it twi}}}\def\pAss{{p\Associative^3}}
\def\udf{{\mathfrak {Udf}}}
\def\Associative{\EuScript{A}{\it ss}}
\def\UComm{\mbox{{${\mathcal U}{\EuScript C}${\it om}}}}
\def\UAss{\mbox{\underline{${\mathcal U}{\EuScript A}${\it s}}}}
\def\ot{\otimes}\def\id{1\!\!1}\def\bbB{{\mathbb B}}\def\bbb{{\mathbb b}}
\def\End{\hbox{${\mathcal E}\hskip -.1em {\it nd}$}}
\def\bEnd{\End^B}\def\bfk{{\mathbb k}}\def\ENd{{\rm End}}
\def\Rada#1#2#3{#1_{#2},\dots,#1_{#3}}
\def\otexp#1#2{#1^{\ot #2}}
\def\calU{{\mathcal U}}\def\calP{{\mathcal P}}\def\calQ{{\mathcal Q}}
\def\calS{{\mathcal S}}\def\Int{\EuScript{I}{\it nt}}
\def\Im{{\rm Im}}\def\cot{{\widehat \ot}}\def\tr{{\it tr\/}}
\def\otred{\hbox{${\hskip -.022em \ot \hskip -.022em}$}}
\def\Def{{\mathfrak{Def}}}

\title{Universal deformation formulas}

\author{Elisabeth Remm, Martin Markl}

\thanks{The second author was supported by the Eduard \v Cech
  Institute P201/12/G028 and RVO: 67985840.}


\catcode`\@=11
\address{Mathematical Institute of the Academy, {\v Z}itn{\'a} 25,
         115 67 Prague 1, The Czech Republic}
\address{MFF UK, 186 75 Sokolovsk\'a 83, Prague 8, The Czech Republic}
\email{markl@math.cas.cz}

\address{Laboratoire de Math\'ematiques et Applications,
        Universit\'e de Haute Alsace, Facult\'e des Sciences et
        Techniques, 4, rue des Fr\`eres Lumi\`ere,
        68093~Mulhouse~cedex, France.}
\email{Elisabeth.Remm@uha.fr}
\catcode`\@=13

\begin{document}
\bibliographystyle{plain}

%\linenumbers

\begin{abstract}
We give a conceptual explanation of universal deformation formulas for
unital associative algebras and prove some results on the structure of their
moduli spaces.  We then generalize universal deformation formulas to
other types of algebras and their diagrams.
\end{abstract}

\maketitle

\tableofcontents

\baselineskip 16pt plus 1pt minus 1 pt

\section{Introduction}

Let $A$ be an associative unital algebra, and $B$ a bialgebra acting
on $A$ in a way that is compatible with the multiplication and unit of
$A$.  A universal deformation formula (abbreviated UDF) is an element
of an extension of $B \otred B$ that determines a deformation of $A$
by `twisting' the original multiplication.

The first explicit UDF was given by exponentiating commuting
derivations.\footnote{We are indebted for historical remarks in this
paragraph to Tony Giaquinto.}  This construction appeared without
proof in M.~Gerstenhaber's \cite{MR0240167}. Its particular
instance when the commuting derivations are $\partial/\partial p$ and 
$\partial/\partial q$ is known today as the Moyal product.  The
first non-abelian UDF was found in 1989 by Coll, Gerstenhaber and
Giaquinto~\cite{CGG} who analyzed a particular quantum group.  In all
these examples, the bialgebra $B$ was the universal enveloping of a
Lie algebra. An example not based on enveloping
algebras can be found in the paper~\cite{CGW} by
C{\u{a}}ld{\u{a}}raru, 
Giaquinto and Witherspoon.

The first aim of this note is to explain why the machinery of
universal deformation formulas works.
Our second aim is to prove some results about bialgebra
actions and moduli spaces of universal deformation formulas. 
Finally, we generalize the apparatus of universal deformation formulas to
other types of algebras and their diagrams.

Since our results concern general algebraic structures, it is not
surprising that language of operads is used. To make the paper more
accessible, we added Appendix~\ref{sec:operad-terminology} as fast
introduction to operads.

Our paper however contains several statements that can be understood
without speaking operads, namely Propositions~\ref{eq:11}
and~\ref{sec:formal-deformations-2}, 
Theorem~\ref{Zitra_prijede_Jaruska.},
Proposition~\ref{dnes_konference_v_Mulhouse},
Corollary~\ref{sec:deform-coming-from}, Remark~\ref{dnes_prileti_Nunyk} and
Proposition~\ref{sec:generalizations}. Most of examples also do not
depend on operads. On the other hand, some constructions described
here might be interesting also from purely operadic point of view,
namely the operad $\bbB$ defined in Section~\ref{sec:an-explanation-5}
and its `logarithmic' version $\bbb$ in
Section~\ref{sec:deformations-udfs}.

\noindent 
{\bf Conventions.}
Most of the algebraic objects will be considered
over a fixed field $\bfk$ of characteristic zero.  
The symbol $\otimes$ will denote the
tensor product over $\bfk$ and $\Span(S)$ the $\bfk$-vector space
spanned by a set $S$. We will
denote by $\id_X$ or simply by $\id$ when $X$ is understood, the identity
endomorphism of an object $X$ (set, vector space, algebra, \&c.).
We will sometimes denote the product of elements $a$ and $b$ using 
the explicit name of the multiplication (typically $\mu(a,b)$),
sometimes, when the meaning is clear from the context, by $a\cdot b$, or
simply by $ab$. We use the same convention also for the module actions.

\section{Recollection of classical results}
\label{sec:recoll-class-results-1}

In this section we briefly recall some constructions of
\cite{giaquinto,gia-zhang}.  Let $B =
(B,\cdot\,,\Delta,1,\epsilon)$ be an unital and counital
associative and coassociative bialgebra over the~ground field $\bfk$,
with the multiplication $\cdot: B \ot B \to B$, comultiplication
$\Delta : B \to B\ot B$, unit $1 \in B$ and counit $\epsilon: B \to
\bfk$. It is a standard fact that the tensor product $M \ot N$ of left
$B$-modules $M$ and $N$ bears a natural left $B$-module action
\[
b (m \ot n) := \sum ( b_{(1)}m \ot b_{(2)}n) \ \mbox { for } \ 
b \in B,\ m \in M, \ n
\in N,  
\]
with Sweedler's notation $\Delta(b) =  b_{(1)} \ot b_{(2)}$.
An associative unital algebra $A = (A,\cdot\, ,1)$ is a (left) {\em
$B$-module algebra\/} if the multiplication $\cdot : A \ot A \to A$
is $B$-linear and $b \cdot 1 = \epsilon(b) \cdot 1$ for each $b \in
B$.\footnote{The condition $b \cdot 1 = \epsilon(b) \cdot 1$ only expresses
  the $B$-linearity of the unit map $\bfk \ni \alpha  \mapsto \alpha 1
  \in A$.}
The following definition, as well as Theorem
\ref{sec:recoll-class-results-2}, can be found in~\cite{gia-zhang}.  

\begin{subequations}
\begin{definition}
\label{sec:recoll-class-results}
A {\em twisting element\/} is an element $F \in B \otred
B$ satisfying
\begin{align}
\label{d1}
[(\Delta \ot \id)(F)](F \ot 1) &= [(\id \ot \Delta)(F)](1 \ot
  F) \mbox { and }
\\
(\epsilon \ot \id)(F) = & \hskip .4em 1  = \hskip .1em (\id \ot \epsilon)(F),
\label{d2}
\end{align}
where $\id : B \to B$ denotes the identity map.
\end{definition}
\end{subequations}

Notice that~(\ref{d1}) is an equation in $\otexp B3$
while~(\ref{d2}) uses the isomorphism $\bfk \otred B \cong B
\cong B \otred \bfk$.

\begin{theorem}
\label{sec:recoll-class-results-2}
If $F \in B \otred B$
is a twisting element as in Definition
\ref{sec:recoll-class-results}, then the formula
\begin{equation}
\label{eq:5}
a * b :=  \mu_A\big(F(a \ot b)\big), \ a,b \in A,
\end{equation}
where $\mu_A  : A \ot A \to A$ is the multiplication, defines
an associative unital structure on~$A$.
\end{theorem}

\section{An explanation}
\label{sec:an-explanation-5}

Theorem~\ref{sec:recoll-class-results-2} of Section
\ref{sec:recoll-class-results-1} has a neat interpretation provided by
diagram~(\ref{eq:16}) at the end of this section. Let us define
objects used in this diagram.

For a bialgebra $B =(B,\cdot  ,\Delta,1,\epsilon)$ as before, 
consider the collection
$\bbB = \{\bbB(n)\}_{n \geq 0}$ with
$\bbB(n) := B^{\ot n}$  equipped with $\circ_i$-operations
\[
\circ_i : \bbB(m) \ot \bbB(n) \to \bbB(m+n-1),\ m,n \geq 0,\ 1 \leq i
\leq m,
\]
given, for $\Rada u1m,\Rada v1n \in B$, by
\begin{eqnarray*}
\lefteqn{
(u_1 \ot \cdots \ot u_m)\circ_i (v_1 \ot \cdots \ot v_n)}
\\
&:=&
u_1 \ot \cdots \ot u_{i-1} \ot \Delta^{n-1}(u_i)(v_1 \ot \cdots \ot v_n) \ot
u_{i+1} \ot \cdots \ot u_m,
\end{eqnarray*}
where
\[
\Delta^{n-1} := (\Delta \ot \id_B^{\ot n-2})(\Delta \ot \id_B^{\ot n-3})
\cdots \Delta : B \to B^{\ot n}
\]
is the coproduct of $B$ iterated $(n-1)$-times, with the convention
that $\Delta^0 := \id_B$ and $\Delta^{-1} := \epsilon : B \to \otexp
B0 =: \bfk$. The operad unit $\id \in \bbB(1) = B$ is the unit $1
\in B$ and  
the  symmetric group $\Sigma_n$ acts on $\bbB(n)$ by
permuting the factors.

In the following proposition we refer to the definition of operads
based on $\circ_i$-operations given for instance in
\cite[Definition~11]{markl:handbook}, for the terminology see Appendix
\ref{sec:operad-terminology}.

\begin{proposition} 
\label{bigB}
The object $\bbB = \{\bbB(n)\}_{n \geq 0}$ described above is a
non-$\Sigma$ operad with unit $1 \in B = \bbB(1)$.
If $B$ is cocommutative, then $\bbB$ is a $\Sigma$-operad.
If $B$ does not have a counit, the above construction leads to $\bbB$
without the $\bbB(0)$-part. 
\end{proposition}

Proposition \ref{bigB} is proved in
Appendix~\ref{Jaruska_dnes_letela_do_Prahy}.  We saw that each
bialgebra determines an operad.  The following proposition describes
operads emerging in this way.

\begin{proposition}
\label{sec:an-explanation-1}
There is a one-to-one correspondence $B \leftrightarrow \calP(1)$ 
between bialgebras $B$ and non-$\Sigma$ operads
$\calP$ such that
$\calP(n) \cong \otexp{\calP(1)}n$ for each $n \geq 0$ and that,
under this identification,
\begin{equation}
\label{eq:15}
\otexp 1n \circ_i b = \otexp 1{i-1} \ot b \ot \otexp 1{n-i} \ \mbox
       {for each} \ b \in \calP(1),\ 1 \leq i \leq n.
\end{equation}
It restricts to a correspondence between
cocommutative bialgebras and $\Sigma$-operads with the above
property. There is a similar correspondence between bialgebras without
counit and operads $\calP$ without the $\calP(0)$-part.
\end{proposition}

In~(\ref{eq:15}), $1 \in \calP(1)$ is the operadic unit, so
$\otexp 1n \in \otexp{\calP(1)}n \cong \calP(n)$, and also the right
hand side of~(\ref{eq:15}) is interpreted as an element of $\calP(n)$.
We prove Proposition~\ref{sec:an-explanation-1} in Appendix
\ref{Jaruska_dnes_letela_do_Prahy}.

In the following proposition, $\UAss = \big\{\UAss(n)\big\}_{n \geq 0}$
is the non-$\Sigma$ operad for unital associative
algebras. It is given by requiring
that $\UAss(n) :=\bfk$ for each $n \geq 0$ and that all $\circ_i :
\bfk \ot \bfk \to \bfk$ are the canonical isomorphisms $\bfk
\ot \bfk \cong \bfk$. 

\begin{subequations}
\begin{proposition}
\label{sec:an-explanation-4}
A twisting element of Definition
\ref{sec:recoll-class-results} is the same as an operad morphism
\begin{equation}
\label{eq:1}
\Phi : \UAss \to \bbB
\end{equation}
such that
\begin{equation}
\label{eq:28}
\Phi : \bfk = \UAss(0) \to \bbB(0) = \bfk \mbox { is the identity.}%
\footnote{As explained in Remark~\ref{fn:1} below, this condition can be omitted.}
\end{equation}
\end{proposition}
\end{subequations}

\begin{proof}
The operad $\UAss$ is generated by $m \in \UAss(2)$ and $1 \in
\UAss(0)$ subject to the relations
\begin{equation}
\label{eq:13}
m \circ_1 m  = m\circ_2 m\  \mbox { and } \
m \circ_1 1 =m \circ_2 1 = \id,  
\end{equation}
where $\id \in \UAss(1)$ denotes the operadic unit.  A
morphisms~(\ref{eq:1}) is therefore specified by
\begin{align}
\label{eq:29}
F := \Phi(m) \in \bbB(2) = B \otred B\ 
\mbox { and }  \
a := \Phi(1) \in \bbB(0) = \bfk
\end{align}
provided the induced map is compatible with the
relations~(\ref{eq:13}), meaning that 
\begin{subequations}
\begin{align}
\label{eq:25}
F\circ_1 F & = F\circ_2 F,\  \mbox { and }
\\
\label{eq:26}
F \circ_1 a & =F \circ_2 a = \id. 
\end{align}
\end{subequations}
It follows from the definition of $\bbB$ that~(\ref{eq:25})
is precisely~(\ref{d1})  while~(\ref{eq:26}) leads to
\begin{equation}
\label{eq:27}
a(\epsilon \ot \id)(F) = 1  = a(\id \ot \epsilon)(F).
\end{equation}         
We finish by noticing that~(\ref{eq:28}) means $a =
1$ when~(\ref{eq:27}) becomes~(\ref{d2}) of 
Definition~\ref{sec:recoll-class-results}. 
\end{proof}

\begin{remark}
\label{fn:1}
Given a morphism $\Phi : \UAss \to \bbB$,
let $F$ and $a$ be as in~(\ref{eq:29}). 
Clearly $a \not=0$ by~(\ref{eq:27}). We may therefore always replace
such a morphism by a rescaled morphism $\Phi'$ defined by $\Phi'(m)
:= \frac 1a \Phi(m)$ and $\Phi'(1) := 1$ which
satisfies~(\ref{eq:28}).
In terms of twistings, general morphisms  $\Phi : \UAss \to \bbB$ twist
the multiplication by~(\ref{eq:5}), but also replace the unit $1 \in A$
by $a\cdot 1$. Equation~(\ref{d2}) or
the equivalent~(\ref{eq:28}) demands that $1 \in A$ is unchanged. 

It is therefore {\em not necessary\/} to assume ~(\ref{eq:28}), as it
always can be achieved by rescaling. Twisting elements will then
precisely be morphisms $\Phi : \UAss \to \bbB$.
\end{remark}

\begin{example}
\label{sec:an-explanation-3}
One always has the trivial twisting element $1\otred 1 \in B\otred B$
that corresponds to the trivial morphism $\tr : \UAss \to \bbB$
sending $1 \in \bfk = \UAss(n)$ to $\otexp 1n \in \otexp Bn = \bbB(n)$
for each $n \geq 0$. It is indeed trivial, but we will need
the map $\tr$ later for
Proposition~\ref{Jarunka_vcera_odletela_do_Prahy}. 
\end{example}

\begin{example}
\label{sec:an-explanation}
Let $B = \bfk[p]$ be the ring of polynomials in $p$, with the standard
bialgebra structure given by the pointwise multiplication of polynomials and
$\Delta(p) := p \otred 1 + 1 \otred p$. Then $\bbB(n) = \bfk[p]^{\ot n}$ is clearly
isomorphic to $\bfk[\Rada u1n]$, via the isomorphism
\[
\bfk[p]^{\ot n} \ni f_1(p) \ot \cdots \ot f_n(p) \longleftrightarrow
f_1(u_1)\cdots f_n(u_n) \in \bfk[\Rada u1n].
\]
With this identification, the diagonal
$\Delta : \bfk[u_1] \to \bfk[u_1,u_2]$ is given by
$\Delta f(u_1) = f(u_1 + u_2)$ and
a twisting element in $B$ is a polynomial $F(u_1,u_2) \in
\bfk[u_1,u_2]$ satisfying the functional equation
\[
F(u_1 + u_2,u_3)F(u_1,u_2) = F(u_1,u_2 + u_3)F(u_2,u_3)
\]
with the `boundary condition' $F(0,u) = F(u,0) = 1$.  This description
can be easily generalized to polynomial algebras in several variables;
we leave the details as an exercise.
\end{example}

Suppose that $M$ is a left $B$-module.
Consider the subcollection $\bEnd_M$  of the endomorphism operad
$\End_M$ consisting of $B$-linear maps
$\phi : M^{\ot n} \to M$, i.e.~the ones satisfying
\[
b\phi(m_1 \ot \cdots \ot m_n)
= \phi\big(\Delta^{n-1}(b)(m_1 \ot \cdots \ot m_n)\big).
\]
for each $b \in B$ and $m_1 \ot \cdots \ot m_n \in M$. 
Notice that the $B$-linearity of a map $\otexp M0 = \bfk \to M$ i.e.~of
an element $m\in M$ means that $\epsilon(b) \cdot m = b
\cdot m$, due to our convention $\Delta^{-1} = \epsilon$.
In Appendix \ref{Jaruska_dnes_letela_do_Prahy} we verify:

\begin{lemma}
\label{sec:an-explanation-2}
The collection $\bEnd_M$ is a non-$\Sigma$ suboperad of the operad
$\End_M$. It is a $\Sigma$-suboperad if $B$ is cocommutative.
\end{lemma}

\begin{proposition}
\label{uz_treti_den_nejde_internet}
A left $B$-module algebra $A = (A,\cdot,1_A)$ is the same as a left
$B$-module $A$ together with an operad morphism
$\alpha : \UAss \to \bEnd_A$.
\end{proposition}

\begin{proof}
The statement is evident. An associative unital algebra $(A, \cdot,
1_A)$ is the same as an operad morphism $\UAss \to \End_A$. Its structure
operations are $B$-linear if and only if the image of this morphism
belongs to $\End_A^B$.
\end{proof}

Let us consider the map of collections $\pi : \bbB \ot \bEnd_A \to
\End_A$ given, for $b_1 \ot \cdots \ot b_n \in \bbB(n)$, $\phi \in
\bEnd_A(n)$ and $a_1, \ldots,
a_n \in A$, by
\[
\pi\big((b_1 \ot \cdots \ot b_n) \ot \phi\big)(a_1\ot \cdots \ot
a_n):=
\phi(b_1 a_1 \ot \cdots\ot b_n a_n).
\]
The proof of the following proposition is a simple direct verification.

\begin{proposition} \label{pi}
The map $\pi : \bbB \ot \bEnd_A \to
\End_A$ is an operad morphism. If $B$ is commutative, the image of
$\pi$ belongs to $\bEnd_A$.
\end{proposition}

\noindent 
{\bf Explanation of Theorem~\ref{sec:recoll-class-results-2}.}  A
twisting element is an operad morphism $\Phi : \UAss \to
\bbB$ by Proposition~\ref{sec:an-explanation-4}.  A left $B$-module
algebra is an operad morphism $\alpha : \UAss \to
\bEnd_A$ by Proposition~\ref{uz_treti_den_nejde_internet}.
We can therefore form the diagram
\begin{equation}
\label{eq:16}
\xymatrix{
\UAss \ar[r]^{\nabla \hskip 1.8em} \ar@/_1em/[1,2]^\rho 
&
\UAss \ot \UAss \ar[r]^{\Phi \ot \alpha}
&
\bbB \ot \bEnd_A \ar[d]^\pi 
\\
&& \End_A
}
\end{equation}
in which $\nabla : \UAss \to \hbox{$\UAss \ot \UAss$}$ is the standard
diagonal of the Hopf operad $\UAss$, see Appendix
\ref{sec:operad-terminology}.\footnote{More usual notation would be
$\Delta$ which we already reserved for the diagonal of $B$.}
The associative multiplication $a,b \mapsto a*b$ of Theorem
\ref{sec:recoll-class-results-2} is precisely the associative algebra
structure given by the composite morphism $\rho : \UAss \to \End_A$.

\section{Deformations and UDF's}
\label{sec:deformations-udfs}

The importance of twisting elements lies in the fact that they
determine deformations. We
start this section by recalling some notions of~\cite{gia-zhang} adapted to
deformations over an arbitrary local complete Noetherian ring $R =
(R,I)$ with the maximal ideal $I$ and residue field $\bfk = R/I$. For
terminology and basic definitions related to deformations over $R$
we refer e.g.~to~\cite[Chapter~1]{markl12:_defor}. 

Prominent examples of $R$ are the ring $\bfk[[t]]$ of
formal power series in variable $t$,  and the
Artin ring $D:= \bfk[t]/(t^2)$ of dual numbers. Deformations over
$\bfk[[t]]$ are called {\em formal\/}, deformations over  $D$ are {\em
  infinitesimal\/} deformations. In the sequel, $\cot$ (resp.~$\cot_R$) 
will denote the completed tensor product over $\bfk$ (resp.~over $R$).

\begin{definition}
\label{sec:formal-deformations1}
A {\em universal deformation formula\/} (UDF) over a~complete
Noetherian ring $R = (R,I)$ with residue field $\bfk$ based on a
bialgebra $B$ is a twisting element $F$ in the bialgebra $(B \cot
R)\cot_R (B \cot R) \cong (B \otred B) \cot R$ that has the form
\begin{equation}
\label{eq:2}
F = 1\otred 1 + F_\circ,
\end{equation}
with some $F_\circ \in (B\ot B) \cot I$.
\end{definition}

It is clear that a UDF leads,
via formula~(\ref{eq:5}),
to a deformation of $A$ with base $R$. Two UDF's $F',F''$ as in Definition
\ref{sec:formal-deformations1} are {\em equivalent\/} if there exists
\[
G = 1 + G_\circ \in B \cot R, 
\]
with some $G_\circ \in B \cot I$ such that
\begin{equation}
\label{eq:7}
F'' = \Delta (G) F' (G^{-1} \otred G^{-1}).
\end{equation}
It is easy to see that equivalent UDF's produce equivalent
deformations.

Let us denote by $\udf_B(R)$ the moduli space of equivalence classes of
UDF's as in Definition~\ref{sec:formal-deformations1}. Observe that
the sets $\udf_B(R)$ assemble into a functor $\udf_B(-)$ from the
category of local complete Noetherian rings to sets, i.e.~of the same
nature as the deformation functor~\cite[Chapter~1]{markl12:_defor}.  

\begin{example}
\label{dnes_prijede_Jaruska}
For $R = \bfk[[t]]$, UDF's based on $B$ are twisting elements $F \in
(B\otred B)[[t]]$ of the form
\[
F = 1\otred 1 + tF_1 + t^2F_2 + \cdots
\]
with some $F_1,F_2,\ldots \in B \otred B$. The element $G$
in~(\ref{eq:7}) in this case has the form
\[
G = 1 + tG_1 + t^2G_2 + \cdots  \in B[[t]]
\]
with $G_1,G_2,\ldots \in B$. Our definition of a UDF therefore
coincides with that of~\cite{gia-zhang}.
\end{example}

The operad structure of $\bbB$ $R$-linearly extends to an operad
structure on the collections $\bbB\cot R$.  In the following statement
which is a direct consequence of definitions and Proposition
\ref{sec:an-explanation-4}, $\id \cot \eta : \bbB\cot R \to \bbB \cot
\bfk \cong \bbB$ denotes the operad map induced by the augmentation
$\eta : R \to R/I \cong \bfk$, and $\tr : \UAss \to \bbB$ is the
trivial twisting element from
Example~\ref{sec:an-explanation-3}.

\begin{proposition}
\label{Jarunka_vcera_odletela_do_Prahy}
Universal deformation formulas over $R$ based on $B$ are the same as
operad morphisms
\[
\Phi : \UAss \to \bbB \cot R
\]
such that $(\id \cot \eta) \circ \Phi = \tr$ and that $\Phi : \bfk =
\UAss(0) \to \bbB(0) \cot R \cong R$ is the inclusion $\bfk \hookrightarrow R$
given by the unit of $R$.
\end{proposition}

\noindent 
{\bf The logarithmic trick.}
\label{sec:deformations-1}
Since $(B \otred B) \cot I$ is a complete algebra, the formal logarithm 
$f := \log(F)$ of~(\ref{eq:2}) exists.  If $B \cot I$ and therefore
also $(B \otred B) \cot I$ is
{\em commutative\/},\footnote{This happens if and only if either $B$ is
commutative or $I^2 = 0$.} applying the logarithm to~(\ref{d1}) 
leads to
\begin{subequations}
\begin{equation}
\label{eq:3}
(\Delta \ot \id)f + f\ot 1 = (\id \ot \Delta)f + 1 \ot f.
\end{equation}
Condition~(\ref{d2}) for $F$ in~(\ref{eq:2}) is satisfied automatically.
The logarithmic version of~(\ref{eq:7}) reads
\begin{equation}
\label{eq:9}
f'' = f' + \Delta(g) - 1 \ot g - g \ot 1,
\end{equation}
\end{subequations}
with some $g \in B \cot I$.

Therefore, if $B\cot I$ is commutative, $\exp(-)$ resp.~$\log(-)$ define an
isomorphism of the set $\udf_B(R)$ 
and the vector space of solutions of ~(\ref{eq:3})
modulo (\ref{eq:9}).
Denote finally by $\twe(B)$ the $\bfk$-vector space
\begin{equation}
\label{eq:10}
\twe(B) := \frac{\mbox {solutions $\xi \in B \ot B$ of $(\Delta \ot
    \id)\xi + \xi\ot
1 = (\id \ot \Delta)\xi + 1 \ot \xi$}} {\mbox {equivalence $\xi'' \sim \xi' +
\Delta(\gamma) - 1 \ot \gamma - \gamma \ot 1$, $\gamma\in B$.}}
\end{equation}
The $I$-linearity of~(\ref{eq:3}) and~(\ref{eq:9}) implies:

\begin{proposition}
\label{eq:11}
If $B \cot I$ is commutative, i.e.~if either $B$ is commutative or
$I^2=0$, then
\[
\udf_B(R) \cong  \twe(B) \cot I.
\]
\end{proposition}

\begin{proposition}
\label{sec:formal-deformations-2}
For any not necessarily commutative bialgebra $B =
(B,\cdot,\Delta,1,\epsilon)$, the vector space $\twe(B)$ defined
in~(\ref{eq:10}) is isomorphic to the second cohomology group
$H^2\big(\Omega (B)\big)$ of the cobar construction $\Omega(B)$.
\end{proposition}

\begin{proof}[Proof of Proposition~\ref{sec:formal-deformations-2}]
Let us recall the definition of the cobar construction. Given a~co\-augmented
coalgebra $C = (C,\Delta,\eta)$, the coaugmentation coideal is the quotient
$\overline C := C/{\rm Im}(\eta : \bfk \to C)$. One has the induced
reduced diagonal $\overline \Delta : \overline C \to \overline C \ot
\overline C$ and defines
\[
\Omega(C) : = \big( T(\uparrow \! \overline C),d\big),
\]
where $T$ is the tensor algebra functor, $\uparrow$ denotes the
suspension and the differential $d$ is the (unique) derivation
whose restriction to $\overline C$ is $\overline \Delta$.

In our case, $C$ is the coalgebra part of $B$, with the coaugmentation
$\eta : \bfk \to B$ given by $\eta(1) := 1$ which splits the counit
map $\epsilon : B \to \bfk$. Therefore, in this particular case, the
coaugmentation coideal $\overline B$ is isomorphic to ${\it Ker}(\epsilon)$.
Observe that $B^{\ot 2}$
canonically splits as
\[
B\ot B \cong (\bfk \!\ot\! B \oplus  B\! \ot\!  \bfk)
\oplus {\overline B}\ot {\overline B}.
\]
We prove that the only solutions to
(\ref{eq:3}) in $(\bfk \otred B
\oplus  B \otred  \bfk)$ are those proportional to $1 \ot 1 \ot 1$. Each
element in $(\bfk \otred B
\oplus  B \otred  \bfk)$ is of the form
\begin{equation}
\label{eq:6}
1 \ot a + b \ot 1 + \alpha (1\ot 1),
\end{equation}
where $a,b \in \overline B$ and $\alpha \in \bfk$. It is easy to see
that~(\ref{eq:3}) applied to~(\ref{eq:6}) leads to
\[
1\ot a \ot 1 + \Delta(b) \ot 1 = 1 \ot \Delta(a) + 1 \ot b \ot 1.
\]
Now, using the obvious fact that, for $u \in \overline B$,
\begin{equation}
\label{eq:8}
\overline \Delta (u) = \Delta(u) - u \ot 1 - 1 \ot u,
\end{equation}
we rewrite the above equation as
\[
2 (1 \ot a \ot 1) + 1\ot 1 \ot a + \overline \Delta(b) \ot 1 =
2 (1 \ot b \ot 1) + b\ot 1 \ot 1 + 1 \ot \overline \Delta(a)
\]
Since ${\it Im}(\overline \Delta) \subset \overline B \otred \overline
B$, the above equation implies that both $a$ and $b$ are proportional
to $1$, which, since we assumed that $a,b \in \overline B$, may happen
only when $a=b=0$.

So the only solutions of the form~(\ref{eq:6}) are proportional to
$1\otred 1$. Taking~(\ref{eq:9}) with $g=1$ shows that this solution is
equivalent to $0$.

We may therefore consider only solutions of~(\ref{eq:3}) that belong to
$\overline B \otred \overline B$. Using~(\ref{eq:8}), one easily sees
that~(\ref{eq:3}) is equivalent to
\[
(\overline \Delta \ot \id)(f) = (\id \ot \overline \Delta)(f)
\]
which means that~$f$ is annihilated by the differential $d$ in the cobar
construction. Likewise, (\ref{eq:9}) for $g \in \overline B$ reads
$f'' = f' + d(g)$ meaning that $f'$ differs from $f''$ by a boundary.
This finishes the proof.
\end{proof}

An immediate consequence of Propositions
\ref{sec:formal-deformations-2} and~\ref{eq:11} is

\begin{theorem}
\label{Zitra_prijede_Jaruska.}
If $B \cot I$ is commutative, then the set\/ $\udf_B(R)$ of equivalence classes of
UDF's over $R$ based on $B$ 
is isomorphic to space $H^2\big(\Omega (B)\big) \cot I$.
\end{theorem}

\begin{example}
Assume that $B$ is the ring of polynomials $\bfk[p_1,p_2,\ldots]$ with
the standard coalgebra structure, i.e.~the free commutative
associative algebra generated by the vector space $P = {\rm
  Span}(p_1,p_2,\ldots)$, with the coalgebra structure determined by
requiring that $P$ is the subspace of primitives.

The calculation of $H^*\big(\Omega (B)\big)$ is in this case
classical. One has the isomorphism
\begin{equation}
\label{clas}
H^*\big(\Omega (B)\big) \cong \Ext^* P
\end{equation}
with the exteriors (Grassmann) algebra
of the vector space $P$. In particular,
$H^2\big(\Omega (B)\big)\cong P \land P$, so the space $\udf_B(R)$ of equivalence
classes of UDF's based on $B$ over an arbitrary local complete
Noetherian ring $R = (R,I)$ is isomorphic to $(P \land P) \cot I$.
\end{example}

Theorem~\ref{Zitra_prijede_Jaruska.} 
in particular describes equivalence classes
of UDF's based on commutative bialgebras. Little is known about the
general case though one may still say something when $B$ is the
universal enveloping algebra $\calU(L)$ of a Lie algebra $L$. 
According to Drinfel'd, 
for each unital solution $r
\in L \otred L$ of the classical Yang-Baxter equation there exists a UDF
over $\bfk[[t]]$ based on $\calU(L)$ whose linear part
$F_1$ equals $r$, see~\cite{gia-zhang} for details. As we will see later in
Proposition~\ref{sec:recoll-class-results-3}, $B$-module algebra
structures for this type of bialgebras have a simple characterization
in terms of derivations.

{}For an operadic interpretation of the logarithmic trick we
need the `logarithmic' version of the operad $\bbB$. It is, by
definition, the collection $\bbb = \{\bbb(n)\}_{n \geq 1}$ with
$\bbb(n) := B^{\ot n}$ equipped with $\circ_i$-operations
\[
\circ_i : \bbb(m) \oplus \bbb(n) \to \bbb(m+n-1)
\]
given, for $\Rada u1m,\Rada v1n \in B$, by
\begin{eqnarray*}
\lefteqn{
(u_1 \ot \cdots \ot u_m)\circ_i (v_1 \ot \cdots \ot v_n)}
\\
&:=&
\big(u_1 \ot u_{i-1} \ot \Delta^{n-1}(u_i) \ot
u_{i+1} \ot \cdots \ot u_m\big) +
\big(1^{\ot (i-1)} \ot   (v_1 \ot \cdots \ot v_n) \ot 1^{\ot (m-i)}\big).
\end{eqnarray*}
The operad unit $\id \in \bbb$ is $0\in B = \bbb(1)$ 
and the action of  $\Sigma_n$ on $\bbb(n)$
permutes the factors.

\begin{proposition}
The object $\bbb = \{\bbb(n)\}_{n \geq 0}$ described above is an
non-$\Sigma$
operad with unit $0 \in  \bbb(1) = B$ in the cartesian monoidal
category of vector spaces with the monoidal structure given by the
direct sum. If $B$ is cocommutative, then $\bbb$ is a $\Sigma$-operad.
\end{proposition}

The proof is similar to that of Proposition \ref{bigB}.  Observe that
the operad $\bbb$ does not use the multiplication in $B$ so that it
makes sense for an arbitrary not necessarily cocommutative coalgebra
with a group-like element $1$.  One may formulate a characterization
of operads of this type analogous to Proposition
\ref{sec:an-explanation-1}, which we leave as an exercise.

\begin{proposition}
\label{sec:deformations}
If $B \cot I$ is commutative, then UDF's
based on $B$ over $R$ are in one-to-one correspondence with  operad morphisms
\[
\phi : \underline{\Associative} \to \bbb \cot I,
\]
where $\underline{\Associative}$ is the non-$\Sigma$
operad for (nonunital) associative algebras.
\end{proposition}

\begin{proof}
The operad
$\underline{\Associative}$ is generated by $m \in
\underline{\Associative}(2)$ such that $m\circ_1 m = m \circ_2 m$.
An operad morphism $\phi : \underline{\Associative} \to \bbb \cot I$
is therefore the same as an element $f := \phi(m) \in \bbb(2) \cot I = (B
\otred B)\cot I$ such that  $f\circ_1 f = f \circ_2 f$. It follows
from the definition of the $\circ_i$-operations in  
$\bbb$ that the last equation is
precisely~(\ref{eq:3}). 
\end{proof}

\section{Examples}

The purpose of this section is to review two well-known examples of
deformations coming from UDF's.  We however start by a useful
characterization of $B$-module algebra structures on a given unital
associative algebra $A$ when $B$ is primitively generated and
cocommutative i.e., by the classical Milnor-Moore theorem,
when $B = \calU (L)$ is the universal enveloping algebra of a Lie
algebra $L$. The following proposition is implicit in~\cite{gia-zhang}.

\begin{proposition}
\label{sec:recoll-class-results-3}
Let $A$ be an associative algebra and $B = \calU (L)$. Then there is a
one-to-one correspondence between $B$-module algebra structures on $A$
and Lie algebra morphisms
\begin{equation}
\label{jdu_si_zabehat}
\alpha : L \to \Der(A)
\end{equation}
from $L$ to the Lie algebra $\Der(A)$ of derivations of $A$.
\end{proposition}

\begin{proof}
A $B$-module structure on $A$ is the same as a
morphism $\calU(L) \to \ENd(A)$ to the
algebra of endomorphisms of the vector space $A$. 
By the universal property of
enveloping algebras, these morphisms are in one-to-one correspondence
with Lie algebra
morphisms $\alpha : L \to \End(A)_{\it Lie}$ to the Lie
algebra associated to $\End(A)$. 

We finish the proof by verifying that the multiplication in $A$ is
linear under the action induced by $\alpha$ if and only if 
\begin{equation}
\label{eq:14}
\Im(\alpha) \subset \Der(A).
\end{equation}

Any $x \in L$ is primitive, i.e.~$\Delta(x) = x \ot 1 + 1 \ot x$, therefore the
$B$-linearity of the multiplication in $A$ requires that
\[
\alpha(x) (a' \cdot a'') =  \alpha(x) (a') \cdot a'' + a' \cdot
\alpha(x) (a'')\ \mbox { for each }  a',a'' \in A,   
\]
i.e.~$\alpha(x) \in \Der(A)$ for each $x\in L$. This shows the
necessity of~(\ref{eq:14}). On the other hand, since each element of $B =
\calU(L)$ is a linear combination of finite products of primitive elements, the 
inclusion~(\ref{eq:14}) is also sufficient. The condition  $b \cdot 1
= \epsilon(b) \cdot 1$ for each $b\in B$ is
for the actions of the above type satisfied automatically.
\end{proof}

\begin{example}
\label{sec:recoll-class-results-31}
There are three important particular instances 
of Proposition \ref{sec:recoll-class-results-3}. 

\noindent 
{\it Case 1: $L$ is abelian.} If $L := \Span(p_1,p_2,\ldots)$ with the trivial
bracket, then $B =
\calU(L)$ is the polynomial algebra $\bfk[p_1,p_2,\ldots]$.  By Proposition
\ref{sec:recoll-class-results-3}, 
$B$-module algebra structures on $A$ are given by mutually
commuting derivations $\theta_1 , \theta_2, \ldots \in \Der(A)$. 

\noindent 
{\it Case 2: $L$ is free.} Let $L := {\mathbb L}(X)$ be the free Lie
algebra on the vector space $X := \Span(e_1,e_2,\ldots)$. Then $B$ is
the tensor algebra $T(X)$ and $B$-module algebra structures on $A$ are
in one-to-one correspondence with (arbitrary) derivations $\theta_1,\theta_2,
\ldots \in \Der(A)$.

\noindent 
{\it Case 3: $L  = \Der(A)$.} In this case one has the `tautological'
$\calU\big(\Der(A)\big)$-module algebra structure given by the identity map $\id :
\Der(A) \to \Der(A)$. 
Each
$\calU(L)$-module algebra action clearly 
uniquely factorizes via this action of $\calU\big(\Der(A)\big)$.
In particular, each twisting element in $\otexp{\calU(L)}2$ 
determines a twisting element in
$\otexp{\calU\big(\Der(A)\big)}2$. 
\end{example}

\begin{example}
Let  $B$ be the `group ring' $\bfk[M]$ of an associative unital monoid
$M$. It is the bialgebra defined precisely as the group ring of a
group, except that it may not have an antipode due to the lack of the
inverse in $M$. 

As in the proof
of Proposition~\ref{sec:recoll-class-results-3} one easily verifies
that $B$-module algebra structure on an associative algebra $A$ is the
same as a morphism of unital monoids
\[
M \to {\it Alg}(A,A),
\]
where ${\it Alg}(A,A)$ is the monoid of algebra morphisms $f : A \to
A$.
\end{example}

\begin{example}
\label{sec:examples}
Very particular examples of formal deformations induced by a UDF are
the Moyal bracket and the quantum plane. In both cases, $B :=
\bfk[p_1,p_2]$, the UDF is given by exponentiating $p_1 \land p_2$,
i.e.
\begin{equation}
\label{maji_na_mne_neco?}
F :=
\exp\left[\frac t2(p_1 \ot p_2 - p_2 \ot p_1)\right]
\end{equation}
and $A := \bfk[p,q]$.  In the case of the {\em Moyal product\/}, $B$
acts on a polynomial $f \in A$ by
\[
p_1 f :=  \frac{\partial f}{\partial p} \ \mbox { and } \
p_2 f :=  \frac{\partial f}{\partial q}.
\]
The Moyal product is therefore the deformation of $A$ given by
\begin{align*}
f * g 
:& = 
\exp\left[\frac t2\left(\frac{\partial \hphantom {f}}{\partial p} \frac{\partial
    \hphantom {g}}{\partial q} -  
\frac{\partial \hphantom {g}}{\partial q} \frac{\partial
    \hphantom {f}}{\partial p}\right)\right](f\ot g)
\\
&= fg + \frac t2\left(\frac{\partial  {f}}{\partial p} \frac{\partial
     {g}}{\partial q} -  
\frac{\partial  {f}}{\partial q} \frac{\partial
     {g}}{\partial p}\right)+
 \frac{t^2}{2^2 \cdot 2!}\left(\frac{\partial^2  {f}}{\partial p^2} \frac{\partial^2
     {g}}{\partial q^2} - 2 
\frac{\partial^2  {f}}{\partial p\partial q} \frac{\partial^2
     {g}}{\partial p\partial q} +  \frac{\partial^2  {f}}{\partial q^2} \frac{\partial^2
     {g}}{\partial p^2} \right) + \cdots
\end{align*}

In the case of the {\em quantum plane\/}, $B$ acts on a polynomial $f
\in A$ as
\begin{equation}
\label{dopustil_jsem_se_neceho?}
p_1 f :=  p \frac{\partial f}{\partial p} \ \mbox { and } \
p_2 f :=  q \frac{\partial f}{\partial q}.
\end{equation}
The product in the quantum plane therefore equals
\[
f * g := \exp\left[t\left( 
p\frac{\partial \hphantom{f}}{\partial p}\ q \frac{\partial
    \hphantom{g}}{\partial q} - q \frac{\partial
\hphantom{g}}{\partial q}\  
p\frac{\partial
    \hphantom{f}}{\partial p}\right)\right] (f \ot g).
\] 
\end{example}

\section{Deformations coming from UDF's.}

In this section we address the size of the set of
deformations coming from UDF's in the moduli space of all
deformations.  For an algebra $A$ denote by $\Def_A(R)$ the moduli
space of deformations of $A$ over $R$. Each universal deformation
formula $F \in \udf_B(R)$ determines, via~(\ref{eq:5}),  
an element in $\Def_A(R)$.  This
gives a map
\begin{equation}
\label{zase_srdce}
\mfI_{B,A}(R): \udf_B(R) \to \Def_A(R)
\end{equation}
which in fact assembles into a transformation $\mfI_{B,A}(-): \udf_B(-)
\to \Def_A(-)$ of functors of local complete Noetherian rings. The 
map~(\ref{zase_srdce}) is of course not an epimorphism in general
-- take as $B$ the trivial bialgebra and as 
$A$ any algebra admitting non-trivial
deformations. There is however an interesting particular case when
$\mfI_{B,A}(R)$ is fully understood. 

Recall that a commutative
associative algebra $A$ is {\em smooth\/} if its module $\Omega^1(A)$
of K\"ahler differentials is projective over $A$. Examples are
coordinate rings of smooth algebraic varieties, see
e.g.~\cite[Chapter~9]{ginzburg}. One has:  

\begin{proposition}
\label{dnes_konference_v_Mulhouse}
Assume that $A$ is a smooth algebra,
$B = \calU\big(\Der(A)\big)$ acting on $A$ as in the 3th case of
Example~\ref{sec:recoll-class-results-31} and $R$ the ring $D$ of dual
numbers. Then~(\ref{zase_srdce}) is
an epimorphism whose kernel equals the kernel of the map
\[
\Der(A)\ \Ext\ \Der(A) \longrightarrow \Der(A)\ \Ext_A\ \Der(A),
\]
where $\Ext$ (resp.~$\Ext_A$) denotes the exterior (wedge) product over $\bfk$
(resp.~over $A$).
\end{proposition}

\begin{proof} 
Propositions~\ref{eq:11} and~\ref{sec:formal-deformations-2} imply
that $\udf_B(D) \cong H^2\big(\Omega(B)\big)$ while the identification
of $ \Def_A(D)$ with the 2nd Hochschild cohomology group $H^2(A,A)$ is
classical~\cite[Theorem~2.3]{markl12:_defor}. So~(\ref{zase_srdce}) is in this case
identified with the map
\[
H^2(\Omega(B)) \longrightarrow H^2(A,A).
\]
By~(\ref{clas}), in the situation of the proposition
$H^2\big(\Omega(B)\big) \cong \Ext^2 \Der(A)$,
while, by the Hochschild-Kostant-Roitenberg theorem \cite[Theorem~9.1.3]{ginzburg},
$H^2(A,A) \cong  \Ext^2_A  \Der(A)$. This finishes the proof.
\end{proof}

\begin{corollary}
\label{sec:deform-coming-from}
Let $A$ be a smooth algebra having two mutually commutative
derivations $\theta_1$ and $\theta_2$ such that
\begin{equation}
\label{eq:30}
\theta_1 \land_A \theta_2 \not= 0 \ \mbox { in } \ \Der(A) \Ext_A \Der(A).
\end{equation}
Then $A$ admits a non-trivial deformation.
\end{corollary}

\begin{proof}
Let $B := \bfk[p_1,p_2]$. The formula $p_ia :=
\theta_i(a)$, $i=1,2$, $a \in A$, defines a $B$-module algebra action on
$A$ and the universal deformation formula 
$F$ from~(\ref{maji_na_mne_neco?}) a deformation of $A$.

Consider the composition
\begin{equation}
\label{internet_stale_nejde}
\udf_B\big(\bfk[[t]]\big) \longrightarrow
\udf_{\calU\Der(A)}\big(\bfk[[t]]\big) \longrightarrow \udf_{\calU\Der(A)}(D)
\longrightarrow \Def_{\calU\Der(A)}(D)
\end{equation}
whose first arrow is induced by the natural map $B = \bfk[p_1,p_2] \to
\calU\Der(A)$ that sends $p_i$ to $\theta_i$, $i =
1,2,$\footnote{Cf.~the universality of the Lie algebra of derivations
mentioned in Case~3 of Example~\ref{sec:recoll-class-results-31}.} the
second map is induced by the ring change $\bfk[[t]] \to \bfk[t]/(t^2) =
D$ and the last map is $\mfI_{\calU\Der(A),A}(D)$.

To prove that the deformation induced by $F$ is non-trivial, it is enough to show
the nontriviality of the 
induced infinitesimal deformation. In other words, we must show that $F$ is
mapped by the composition~(\ref{internet_stale_nejde}) 
to a non-trivial deformation.  By
Proposition~\ref{dnes_konference_v_Mulhouse} this means 
that $\theta_1 \land_A \theta_2$ is non-zero in $\Ext^2_A \Der(A)$. 
\end{proof}

\begin{remark}
\label{dnes_prileti_Nunyk}
Corollary \ref{sec:deform-coming-from} generalizes to an arbitrary
number of commuting derivations; we leave the
details as an exercise. The method of its proof can also be used for 
showing that a given deformation coming from an UDF (not necessarily based on a
commutative bialgebra) is nontrivial, or to prove that two or more
deformations of this type are non-equivalent. 
\end{remark}

Condition~(\ref{eq:30})  in Corollary \ref{sec:deform-coming-from} is
substantial. If
\[
A =
\bfk[p,q],\ \theta_1 := 
\frac{\partial \hphantom p}{\partial p} \ \mbox { and }
\ \theta_2 :=
q\frac{\partial \hphantom p}{\partial p},
\] 
then the induced deformation is trivial although $\theta_1 \not =
\theta_2 \not= 0$. 
Also the assumption that $A$ is smooth is
important. If e.g.
\[
A := \frac{\bfk[p,q]}{(p^2 = q^2 = pq = 0)},
\] 
then $\theta_1 := p{\partial
  }/{\partial p}$, $\theta_2 :=q{\partial
  }/{\partial q}$ are commuting derivations, $\theta_1 \land_A \theta_2
  \not= 0$, but the induced deformation is trivial.

Corollary~\ref{dnes_konference_v_Mulhouse} is an algebraic version of
the claim that the algebra of smooth functions on a smooth manifold of
dimension $\geq 2$ admits a non-trivial deformation, see~\cite{bordemann:defquant}.

\section{Generalizations}

The situation described by the diagram in~(\ref{eq:16}) generalizes as
follows. Suppose we have the following morphisms of non-$\Sigma$-operads:
\begin{itemize}
\item[(i)]
a `diagonal'
$\nabla : \mathcal{P} \to \hbox{$\mathcal{Q} \ot \mathcal{R}$}$,
\item[(ii)]
a `$\calQ$-twisting element' $\Phi : \mathcal{Q} \to \bbB$ and
\item[(iii)]
a $B$-module $\calR$-algebra structure $\alpha : \mathcal{R} \to \bEnd_A$.
\end{itemize}
Then the diagram
\begin{equation}
\label{eq:18}
\xymatrix{
\calP \ar[r]^{\nabla \hskip 1.4em} \ar@/_1em/[1,2]^\rho 
&
\calQ \ot \calR \ar[r]^{\Phi \ot \alpha \hskip .8em}
&
\bbB \ot \bEnd_A \ar[d]^\pi 
\\
&& \End_A
}
\end{equation}
determines  a $\calP$-algebra structure on $A$ via the composite map
$\rho : \calP \to \End_A$.

As in the classical case, we want to use $\calQ$-twisting
elements to deform $\calP$-algebras. To do so we
need a good notion of the trivial $\calQ$-twisting element, i.e.~the
one that does not twist. A moment's reflection tells us that this
amounts to a distinguished operad morphism
\begin{equation}
\label{eq:17}
\omega : \calQ \to \UAss 
\end{equation}
which defines the trivial $\calQ$-twisting element as the composition
\[
\Phi_0 : \calQ 
\stackrel{\omega}\longrightarrow \UAss \stackrel{\tr}{\longrightarrow} \bbB,
\]
where $\tr$ is the trivial $\UAss$\,-twisting element 
from Example~\ref{sec:an-explanation-3}.  The
`untwisted'
$\calP$-algebra structure on $A$ is then given by the composition
\begin{equation}
\label{eq:23}
\calP \stackrel{\nabla}{\longrightarrow} \calQ \ot \calR
\stackrel{\omega \ot \id}{\longrightarrow} \UAss \ot \calR \cong \calR
\stackrel{\alpha}{\longrightarrow} \bEnd_A. 
\end{equation}

One may then define 
a $\calQ$-universal deformation formula (a~$\calQ$-UDF)
based on $B$ over $R$ as an operad morphism
\[
\Phi : \calQ \to \bbB \cot R
\]
such that $(\id \cot \eta) \circ \Phi = \Phi_0$. It leads,
via~(\ref{eq:18}), to a deformation of the untwisted $\calP$-algebra
structure~(\ref{eq:23}). We will see that in many situations of
interest there exists a canonical $\omega$ in~(\ref{eq:17}).  

\noindent {\bf Variants.}
If $B$ is cocommutative, one may consider versions of the above
objects in the category of $\Sigma$-operads. Instead of $\UAss$
in~(\ref{eq:17}) one then takes the operad $\UComm$ for unital
commutative associative algebras. For the study of structures without
constants (such as e.g.~Lie algebras), it is enough to 
work in the category of operads
without the arity $0$-piece. In this case, the bialgebra $B$ need not
have a counit, cf.~Proposition~\ref{bigB}.

There are surprisingly many situations where the above data
exist. First of all, $\UAss$ is the unit in the monoidal category of
non-$\Sigma$ operads, therefore for each non-$\Sigma$ operad $\calP$
one has the canonical `diagonal' $\nabla : \calP \to \UAss \ot \calP$
given by the isomorphisms 
\[(\UAss \ot \calP)(n) = \UAss(n) \ot \calP(n)
= \bfk \ot \calP(n) \cong \calP(n),\ n \geq 0.
\] 
We therefore have:

\begin{proposition}
\label{sec:generalizations}
The usual twisting element from Definition
\ref{sec:recoll-class-results} 
twists algebras over an arbitrary
non-$\Sigma$ operad $\calP$, and UDF's as in Definition
\ref{sec:formal-deformations1} determine their deformations.
\end{proposition}

Similarly, $\UComm$ is the unit for $\Sigma$-operads, so for any
$\Sigma$-operad $\calP$ we have the diagonal $\nabla : \calP \to
\UComm \ot \calP$. It is easy to verify that a $\UComm$-twisting element
is a twisting element as in  Definition \ref{sec:recoll-class-results}
that is {\em symmetric\/}, i.e.~$F = \tau F$, where $\tau : B \otred B
\to B \otred B$ interchanges the factors. Twisting
elements with this property twist and deform algebras over an {\em arbitrary\/}
operad, i.e.~almost all `reasonable' algebras.

\begin{example}
\label{sec:generalizations-2}
There is a particularly important instance of the situation described
in the above paragraph: $\calP = \Lie$, the operad for Lie
algebras. Since Lie algebras are structures without constants, we may
work in the category of $\Sigma$-operads without the arity
$0$-part and consider the Lie version of the diagram~(\ref{eq:16}):
\[
\xymatrix{
\Lie \ar[r]^{\nabla \hskip 1.7em} \ar@/_1em/[1,2]^\rho 
&
\Com \ot \Lie \ar[r]^{\Phi \ot \alpha}
&
\bbB \ot \bEnd_L \ar[d]^\pi 
\\
&& \End_A
}
\]
in which 
\begin{itemize}
\item[(i)]
$\Com$ is the operad for commutative associative algebras and
$\nabla : \Lie \to \Com \ot \Lie$ the canonical diagonal,
\item[(ii)]
$\Phi : \Com \to \bbB$ is a symmetric twisting element and
\item[(iii)]
$\alpha : \Lie \to \bEnd_A$ is a given $B$-module Lie algebra
structure on $L$.
\end{itemize}
Explicitly, (ii) requires an element $F = \sum_i F^{(1)}_i \otred
F^{(2)}_i  \in \otexp B2$, for a possibly non-counital, cocommutative bialgebra
$B$, such that
\begin{equation}
\label{eq:24}
\begin{aligned}
{}[(\Delta \ot \id)(F)](F \ot 1) &= [(\id \ot \Delta)(F)](1 \ot F) \mbox
  { and } 
\\ 
\textstyle
\sum_i F^{(1)}_i \otred F^{(2)}_i &=\textstyle \sum_i F^{(2)}_i
  \otred F^{(1)}_i.
\end{aligned}
\end{equation}
Observe that~(\ref{d2}) is not required.
\end{example}

\noindent 
{\bf Set-theoretic case.}  As noticed in~\cite{markl-remm:JA06} and recalled in
Appendix~\ref{sec:operad-terminology}, each {\em set-theoretic\/}
operad $\calP$, i.e.~one that is a linearization of an operad defined
in category of sets, is a Hopf operad with the canonical diagonal
$\nabla : \calP \to \calP \ot \calP$. The prominent example is the
operad $\UAss$ for unital associative algebras. Since $\UAss$ is the
terminal object in the category of set-theoretic operads, it is
equipped with the canonical morphism $\omega : \calP \to
\UAss$. Another useful feature of the set-theoretic case is that the
logarithmic trick described on page~\pageref{sec:deformations-1} is
available.  Namely, one has an analog of
Proposition~\ref{sec:deformations}:

\begin{proposition}
If $\calP$ is a set-theoretic operad and $B \cot I$ commutative, then $\calP$-UDF's
based on B over $R$ are in one-to-one correspondence with operad morphism
$\calP \to \bbb \cot I$.
\end{proposition}

\begin{proof}
By assumption,  $\calP \cong \bfk \otred \calS$ for
some operad $\calS$ in the category of sets. Operad maps $\calP \to
\bbB \cot R$ are then the same as operad maps $\calS \to
\bbB \cot R$, where $\bbB \cot R$ is considered as an operad in the
category of sets. In the last description, only the multiplication,
not the addition, of
$B$ is used, thus the logarithmic trick applies. 
\end{proof}

Several other examples of diagonals $\nabla : \mathcal{P} \to
\hbox{$\mathcal{Q} \ot \mathcal{R}$}$, sometimes very exotic, can be found in
\cite{goze-remm:tensor}. 
As far as $B$-module $\calR$-algebra structures are concerned,  
one has an obvious analog of
Proposition~\ref{sec:recoll-class-results-3}, with the map~(\ref{jdu_si_zabehat})
replaced by $\alpha : L \to \Der_\calR(A)$,
where $\Der_\calR(A)$ denotes the Lie algebra of derivations of the
$\calR$-algebra $A$.

\begin{example}
\label{sec:generalizations-1}
A {\em partially associative ternary algebra\/} $A$ is a structure with a ternary
operation $A \ni a,b,c \mapsto (a,b,c) \in A$ such that
\[
\big(a,b,(c,d,e)\big) + \big(a,(b,c,d),e\big) + \big((a,b,c),d,e\big) =0.
\]
We use this randomly chosen type of algebras to
illustrate Proposition~\ref{sec:generalizations} and demonstrate how our
theory applies to algebras having other than binary
operations. We could as well use 
other types of algebras over non-$\Sigma$-operads, as
diassociative, dipterous, dendriform, duplicial, Sabinin, 
totally associative ternary, \&.,  see~\cite{zinbiel11:_encyc} for
definitions.

Partially associative ternary algebras are algebras over the
non-$\Sigma$-operad $\pAss$~\cite{markl-remm}. 
To see how an ordinary twisting element determines a $\pAss$-twisting,
we need to track the maps in diagram~(\ref{eq:18}) with $\calP = \calR
= \pAss$ and $\calQ = \UAss$, i.e.~the composition
\begin{equation}
\label{eq:19}
\pAss \stackrel\nabla\longrightarrow \UAss \ot \pAss 
\stackrel{\Phi \ot \alpha}\longrightarrow \bbB \ot \bEnd_A
\stackrel{\pi}\longrightarrow \End_A.
\end{equation}
The first map $\nabla$ takes the generator $\nu \in \pAss(3)$ for the
ternary operation $(-,-,-)$ to $\mu_3 \ot \nu$, where
\[
\mu_3 := \mu_2 \circ_1 \mu_2 = \mu_2 \circ_2 \mu_2 \in \UAss(3).
\]
The element $H := \Phi(\mu_3) \in \otexp B3$ is determined by the
usual twisting element $F = \Phi(\mu) \in \otexp \bbB2$ by the formula
\begin{equation}
\label{eq:20}
H = \Phi(\mu_3) = \Phi(\mu_2 \circ_1 \mu_2) = F \circ_1 F = 
[(\Delta \ot \id)(F)](F \ot 1).
\end{equation}

To see how $H$ twists $(a,b,c)$, we need to finish
reading~(\ref{eq:19}). Invoking the definition of~$\pi$, we easily
get the formula for the twisted ternary product $(a,b,c)'$, namely
\[
(a,b,c)' =\textstyle \sum_i (H^{(1)}_i a, H^{(2)}_i b, H^{(3)}_i c) 
\]
where we wrote  $H = \sum_i H^{(1)}_i \otred H^{(2)}_i \otred H^{(3)}_i$.

Taking $B = \bfk[p_1,p_2]$ and $F = \exp\big[t(p_1 \otred p_2 - p_2
\otred p_1)\big]$ as in Example~\ref{sec:examples}, (\ref{eq:20})
gives the $\pAss$-universal deformation formula
\begin{equation}
\label{eq:21}
H = \exp\big[t\{(p_1 \otred p_2 - p_2\otred p_1) \otred 1  
+  p_1 \otred 1 \otred p_2 - 
p_2 \otred 1 \otred p_1  + 1 \otred  (p_1 \otred p_2 - p_2\otred p_1)\}\big].
\end{equation}
\end{example}

\begin{example}
The UDF~(\ref{eq:21}) from Example
\ref{sec:generalizations-1} leads to deformations via $B$-module
$\pAss$-algebra actions.  Let us give an example of such an action.
A ternary algebra $A$ is {\em symmetric\/}~if
\[
(a_1,a_2,a_3) = (a_{\sigma(1)},a_{\sigma(2)},a_{\sigma(3)}), 
\] 
for any $a_1,a_2,a_3 \in A$ and a permutation $\sigma \in \Sigma_3$.
Let $\bfk^{(p3)}[p,q]$ be the free symmetric
partially associative ternary algebra\footnote{The algebra $\bfk^{(p3)}[p,q]$ as
  well as its deformation that we define below can
  be described very explicitly.}  and $p^2 \frac{\partial \hphantom {p}}{\partial
  p}$ resp.~$q^2 \frac{\partial \hphantom {q}}{\partial  q}$ derivations of
$\bfk^{(p3)}[p,q]$ determined~by
\[
p^2 \frac{\partial\hphantom {p}}{\partial  p} (p) := (p,p,p) , \ 
p^2 \frac{\partial\hphantom {p}}{\partial  p} (q) := 0, \
q^2 \frac{\partial\hphantom {q}}{\partial  q} (p) := 0, \ \mbox { and }\ 
q^2 \frac{\partial\hphantom {q}}{\partial  q} (q) := (q,q,q).
\]
The derivations $p^2 \frac{\partial}{\partial p}$ and $q^2
\frac{\partial}{\partial  q}$ commute with each other, so they define
a $\bfk[p_1,p_2]$-module $\pAss$-algebra action on $\bfk^{(p3)}[p,q]$ by 
\[
p_1a := p^2 \frac{\partial \hphantom {p}}{\partial p} (a) \ \mbox { and }  
p_2a := q^2 \frac{\partial\hphantom {q}}{\partial q} (a),\ a \in  \bfk^{(p3)}[p,q].
\]
Via this action,~(\ref{eq:21}) gives a (non-symmetric)
deformation of the ternary algebra $\bfk^{(p3)}[p,q]$ which is a ternary
analog of the quantum plane. 

The above construction may seem formal and shallow, but
its complexity and non-triviality is the same as that of the quantum
plane, we only worked in the category of ternary partially associative
algebras instead of the category of associative algebras.
A similar `$\calP$-quantum plane' can be
constructed for an arbitrary non-$\Sigma$ operad. We leave the details
to the reader.
\end{example}

\begin{example}
This example illustrates the set-theoretic case. An
{\em interchange algebra\/} has two binary operations
$\bullet, \circ : A \ot A \to A$ that satisfy
\[
(x\bullet y)\circ (z\bullet t) = (x\circ z)\bullet (y\circ t),
\]
see e.g.~\cite{zinbiel11:_encyc}. Since the corresponding operad $\Int$ is
set-theoretic, one may take in~(\ref{eq:18}) $\calP = \calQ
= \calR =: \Int$.
By expanding the definition of $\bbB$ one easily sees that an
$\Int$-twisting cocycle $\Phi : \Int \to \bbB$ is given by a couple
$(F',F'')$ with 
$F',F'' \in B \otred B$ such that
\begin{equation}
\label{eq:22}
\tau_{1324}\{(\Delta \ot \Delta)(F')(F'' \ot F'')\}
=
(\Delta \ot \Delta)(F'')(F' \ot F'),
\end{equation}
where $\tau_{1324}$ permutes the factors of $\otexp B4$ according to
the permutation $(1,2,3,4) \mapsto (1,3,2,4)$.

An example of $(F',F'')$ solving~(\ref{eq:22}) is provided by
\[
F' := a \ot b,\ F'' := c \ot d,
\]
where $a,b,c,d \in B$ are grouplike elements.
Other types of algebras over set-theoretic operads are listed
in~\cite{zinbiel11:_encyc}. 
\end{example}

\noindent 
{\bf Diagrams}. Our theory generalizes to {\em diagrams\/} 
of algebras. Recall
that a diagram of $\calP$-algebras is a functor 
$\A:\Diag \to \hbox{$\calP$-{\it alg}}$
from a small category $\Diag$ encoding the shape of the diagram, into the
category of $\calP$-algebras. Diagrams of $\calP$-algebras are algebras over a
certain colored operad $\calP_\Diag$ whose set of colors $C$ equals
the set of objects (nodes) of $\Diag$. The details can be found
in~\cite{markl:ha}. 

Diagrams $\A$ as above are acted on by
the opposite diagrams $\B: \Diag^{\it op} \to {\it Bialg}$ of
bialgebras. Every such a diagram $\B$ defines a $C$-colored version
$\bbB_\Diag$ of the operad $\bbB$ whose explicit construction we leave
as an exercise.

Rather than to explain general theory, we consider the special case
when $\calP = \UAss$. 
Assume we have a diagram $\A$ of associative unital algebras and an
opposite diagram $\B$ of bialgebras. We say that $\A$ is a (left)
$\B$-module diagram of associative algebras if

\begin{itemize}
\item[(i)] 
for each $v \in \Diag$, $A_v$ is a left
$B_v$-module algebra, where $A_v := \A(v)$ and $B_v := \B(v)$, 
\item[(ii)] 
for each arrow $v' \stackrel r\to v''$ in $\Diag$, each $a \in
A_{v'}$ and $b \in B_{v''}$, 
\[
b\hh_r(a)  =\hh_r\big(\phi_r(b)a\big),
\]
where $\hh_r := \A(r)$ and $\phi_r := \B(r)$.
\end{itemize}

In this situation we have a colored version of the 
diagram~(\ref{eq:18}) with $\calP
= \calQ = \calR := \UAss_\Diag$, instead of $\bbB$ the colored operad $\bbB_\Diag$
and instead of $\End_A$ resp.~$\bEnd_A$ the corresponding colored 
endomorphism operads.

\begin{example}
If $\Diag := 1 \to 2$, then 
$\A:\Diag \to \hbox{$\calP$-{\it alg}}$ is precisely a morphism $\hh:
A_1 \to A_2$ of associative algebras and 
$\B : \Diag^{\it op} \to {\it Bialg}$ is a morphism $\phi :
B_2 \to B_1$ of bialgebras. Saying that $\A$ is a $\B$-module diagram
of associative algebras in this particular case means that 
$A_i$ is a $B_i$-module algebra, $i = 1,2$,
and $b\hh(a) = \hh\big(\phi(a)b\big)$ for each $a \in A_1$, $b \in B_2$.

Twisting element for such a situation is a triple $(F_1,G,F_2)$ such
that

\begin{itemize}
\item[(i)] $F_i \in \otexp{B_i}2$ is a twisting element based on $B_i$
as in Definition
\ref{sec:recoll-class-results}, $i = 1,2$, and
\item[(ii)] 
$G \in B_1$ satisfies
$\Delta(G) F_1 = (\phi \ot \phi)(F_2) (G \ot G)$. 
\end{itemize}

Such a triple  $(F_1,G,F_2)$ has the property that if
$(A_i,*_i)$ is the algebra $A_i$ with the multiplication twisted
by $F_i$ as in Theorem \ref{sec:recoll-class-results-2}, $i=1,2$, 
then $\tilde \hh : A_1 \to A_2$  defined by 
\[
\tilde \hh (a) := \hh(G a),\ a \in A,
\]
is a morphism of these twisted algebras.

Notice that if $G$ is invertible,
one can rewrite (ii) into
\[
\Delta(G) F_1  (G^{-1} \ot G^{-1})  = (\phi \ot \phi)(F_2)
\]
and replace the triple $(F_1,G,F_2)$ by the gauge
equivalent $(F'_1,1,F_2)$, with 
\[
F'_1 := \Delta(G) F_1  (G^{-1} \ot G^{-1}).
\] 
In other words, in the case of a single arrow,
each deformation coming from a UDF is gauge equivalent to one with $\hh$
unperturbed.  

As a very explicit example, let $\A$ be the diagram of associative
algebras with $A_1=A_2 := \bfk[p,q]$ and $\hh : \bfk[p,q] \to \bfk[p,q]$
the morphism given by $\hh(p) := p^m$, $\hh(q) := q^m$, for some integers
$m,n \geq 1$. Let $\B$ be the diagram of bialgebras given by $B_1 =
B_2 := \bfk[p_1,p_2]$ and $\phi : B_2 \to B_1$ be the identity.  

Let $B_1$ acts on $A_1$ as in the case of the quantum plane recalled
in Example~\ref{sec:examples}, i.e.~by
formula~(\ref{dopustil_jsem_se_neceho?}), and $B_2$ acts on $A_2$ by
\[
p_1 f :=  \frac{p^m}{m}\frac{\partial f}{\partial p} \ \mbox { and } \
p_2 f :=  \frac{q^n}{n}\frac{\partial f}{\partial q}.
\]
The above
data determine a $\B$-module structure on the diagram $\A$ and the triple
\[
\Big(\exp\big[t(p_1 \ot p_2 - p_2 \ot p_1)\big],
1,\exp\big[t(p_1 \ot p_2 - p_2 \ot p_1)\big]\Big)
\]
is a universal deformation formula based on $\B$.

The corresponding deformed  $A_1$ is the quantum plane 
and the product in the deformed $A_2$ is given by
\[
f * g := \exp\left[t\left(  p^m
\frac{\partial  \hphantom{f}}{\partial p} \ q^n  \frac{\partial
    \hphantom{g}}{\partial q} -  q^n \frac{\partial
  \hphantom{g}}{\partial q}\  p^m
\frac{\partial
    \hphantom{f}}{\partial p}\right)\right] (f \ot g).
\] 
The undeformed $\hh : A_1 \to A_2$ is a monomorphism of deformed algebras which an
isomorphism if and only if $(m,n) = (1,1)$.  
\end{example}

We saw in the previous example that deformations of morphisms of
algebras coming from UDF's are gauge equivalent to a deformation with
the undeformed morphism. This is not very surprising since the
one-arrow diagram is topologically trivial. For e.g.~$\Diag$ the
triangle
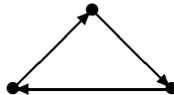
\[
{% Picture saved by xtexcad 2.4
\unitlength=1.000000pt
\begin{picture}(60.00,30.00)(0.00,5.00)
\thicklines
\put(0.00,0.00){\makebox(0.00,0.00){$\bullet$}}
\put(60.00,0.00){\makebox(0.00,0.00){$\bullet$}}
\put(30.00,30.00){\makebox(0.00,0.00){$\bullet$}}
\put(60.00,0.00){\vector(-1,0){60.00}}
\put(30.00,30.00){\vector(1,-1){30.00}}
\put(0.00,0.00){\vector(1,1){30.00}}
\end{picture}}
\]
such a reduction is not possible.

\section*{Open problems}

\noindent
{\bf Problem 1.}
Does there exist a nontrivial symmetric universal twisting formula
based on a cocommutative bialgebra,
that is, an element $F \in \otexp B2$  of the form~(\ref{eq:2})
satisfying equations~(\ref{eq:24})?
As explained in the paragraph following
Proposition~\ref{sec:generalizations}, such a UDF would deform all
`reasonable' algebras, including Lie and Poisson algebras. 
The Lie case is analyzed in detail in Example~\ref{sec:generalizations-2}.

\noindent
{\bf Problem 2.}
Does each deformation of an associative unital algebra $A$
come from a universal deformation formula?
I.e.,~is it true that for a given ring $R$ and an $R$-deformation 
of $A$ there
exists a bialgebra $B$ such that this deformation belongs to the image
of the map~(\ref{zase_srdce})?

We do not think that the above statement is true but we were unable to
find a counterexample. A more specific question is whether deformations
coming from universal deformation formulas have some characteristic
property that distinguishes them from arbitrary ones.

Notice that the unitality is crucial: let $A$ be an arbitrary
non-unital associative algebra with the product $\mu : \otexp A2 \to
A$. Then $t \mu$ is a formal deformation of the trivial multiplication
(the zero map) $\otexp A2 \to A$ which does not come from any
universal deformation formula.

\appendix

\section{The remaining proofs}
\label{Jaruska_dnes_letela_do_Prahy}

\begin{proof}[Proof of Proposition \ref{bigB}]
We need to verify that $\bbB$ fulfills the unitality and associativity
axioms for operads, see e.g. \cite[Definition~11]{markl:handbook}. The unitality
is obvious. Recall that the associativity means that 
for each $1 \leq j \leq a$, $b,c \geq 0$, $u\in
\bbB(a) = \otexp Ba$, $v \in \bbB(b) = \otexp Bb$ and $w \in \bbB(c)= \otexp Bc$,
\[
(u \circ_j v)\circ_i w =
\left\{
\begin{array}{ll}
(u \circ_i w) \circ_{j+c-1}v,& \mbox{for } 1\leq i< j,
\\
u \circ_j(v \circ_{i-j+1} w),& \mbox{for }
j\leq i~< b+j, \mbox{ and}
\\
(u \circ_{i-b+1}w) \circ_j v,& \mbox{for }
j+b\leq i\leq a+b-1.
\end{array}
\right.
\]
It clear that the 1st and 3rd cases of the above axiom hold
even for a non-associative $\Delta$. As for the 2nd case, let $u
= u_1 \ot \cdots u_a$. It is not difficult to compute that
\begin{align*}
(u \circ_j v)\circ_i w &=u_1 \ot \cdots \ot u_{i-1} \ot X'\cdot Y \ot u_{i+1} \ot 
\cdots \ot u_a,
\\
u \circ_j(v \circ_{i-j+1} w)  &
=u_1 \ot \cdots \ot u_{i-1} \ot X''\cdot Y \ot u_{i+1} \ot 
\cdots \ot u_a
\end{align*}
where $X',X'', Y \in \otexp B{b+c-1}$ are given by
\begin{align*}
X' := 
\Delta^{b+c-2}(u_j),\ &
X'' : =
(\otexp \id{i-1} \ot \Delta^{c-1} \ot
\otexp\id{b-i})\Delta^{b-1}(u_j), \mbox { and }
\\
Y := &
(\otexp \id{i-1} \ot \Delta^{c-1} \ot
\otexp\id{b-i})(v).
\end{align*}
The second case of the associativity is therefore satisfied if
$X' = X''$, i.e.~if
\begin{equation}
\label{eq:12}
\Delta^{b+c-2} =  (\otexp \id{i-1} \ot \Delta^{c-1} \ot
\otexp\id{b-i})\Delta^{b-1}.
\end{equation}
But the last equation is clearly implied by the coassociativity of $\Delta$.
Observe that~(\ref{eq:12}) for $b=c=2$, $i=1,2$ gives
$(\Delta \ot \id)\Delta = \Delta^2 = (\id \ot \Delta)\Delta$, so the
coassociativity of $\Delta$ is also necessary for $\bbB$ to be a
non-$\Sigma$ operad.

To prove the second part, one needs to verify the equivariance of the
$\circ_i$-product, cf. \cite[Definition~11]{markl:handbook} again. We
are not going to do it in full generality since it is straightforward
though notationally challenging. We only analyze one particular case
of the equivariance axiom which explains why $\Delta$ must be
cocommutative for $\bbB$ to be $\Sigma$-operad, namely
\[
u \circ_1 (v \tau) = (u \circ_1 v)\tau
\]
for $u\in \bbB(1)$, $v \in \bbB(2)$ and $\tau \in \Sigma_2$ the
transposition. With $v = v_1 \ot v_2$ this axiom reads
\[
\Delta (u) (v_2 \ot v_1) = \tau \Delta(u) (v_2 \ot v_1),
\]
where $\tau \Delta$ is the transposed diagonal. It
is fulfilled if $\Delta = \tau \Delta$, i.e.~if $\Delta$ is
cocommutative. 
\end{proof}

\begin{proof}[Proof of Proposition~\ref{sec:an-explanation-1}]
The operad $\bbB$ of Proposition~\ref{bigB} determined by $B$
obviously has the required properties. On the other
hand, for $\calP$ as in the proposition put $B:= \calP(1)$ with the
unital associative multiplication $\circ_1 : B \otred B \to B$, the
comultiplication given by the formula $\Delta(b) := b\circ_1 (1\ot 1)$
for $b \in B$, and the counit
\[
\epsilon: B \cong B \ot \bfk
\cong \calP(1) \ot \calP(0)   \stackrel{\circ_1}{\to} \calP(0) \cong \bfk.
\]
It is straightforward to verify that the bialgebra axioms are satisfied.
\end{proof}

\begin{proof}[Proof of Lemma~\ref{sec:an-explanation-2}]
It is simple to prove that a composition of $B$-linear maps is
$B$-linear. This proves that $\bEnd_M$ is a non-$\Sigma$ suboperad of $\End_M$.

To prove it is a $\Sigma$-suboperad, we must also verify that the
subspace of $B$-linear maps is closed under the symmetric group
action. Let $\phi \in \bEnd_M(n)$ and $\sigma \in \Sigma_n$. Then, for
each $m_1,\ldots,m_n \in M$ and $b \in B$,
\[
b (\phi\sigma)(m_1 \otimes \cdots \ot m_n) = b\phi(m_{\sigma(1)} \otimes
\cdots \ot m_{\sigma(n)}) = \phi\big(\Delta^{n-1}(b)(m_{\sigma(1)} \otimes
\cdots \ot m_{\sigma(n)})\big)
\]
by the $B$-linearity of $\phi$ and the definition of the
$\Sigma_n$-action. On the other hand, the $B$-linearity of $\phi\sigma$
means that
\[
b (\phi\sigma)(m_1 \otimes \cdots \ot m_n)=
(\phi\sigma)\big(\Delta^{n-1}(b)(m_{\sigma(1)} \otimes
\cdots \ot m_{\sigma(n)})\big)  =
\phi\big(\sigma\Delta^{n-1}(b)(m_{\sigma(1)} \otimes
\cdots \ot m_{\sigma(n)})\big)
\]
here $\sigma\Delta^{n-1}$ denotes $\Delta^{n-1}$ with the coordinates
permuted according to $\sigma$. If $\Delta$ is cocommutative, 
$\sigma\Delta^{n-1} = \Delta^{n-1}$, so $\phi\sigma$ is $B$-linear as well.
\end{proof}

\section{Operad lingo}
\label{sec:operad-terminology}

In this part of the appendix we recall most basic notions of the
operad theory in hope to make the paper readable without their
preliminary knowledge.  There exist a rich and easily accessible literature
devoted to operads, for instance the monograph
\cite{markl-shnider-stasheff:book}, overview articles
\cite{getzler:operads-revisited,markl:pokroky,markl:handbook} or a recent
account~\cite{loday-vallette}.

Operads are devices describing algebraic structures. Roughly, each
algebra whose axioms are homogeneous in the number of variables and do
not use repeated variables and quantifiers, is an algebra over a
certain operad.
Most of algebras one meets in everyday life, such as associative,
commutative, Lie, Poisson, \&c, are therefore algebras over operads.
 
More specifically, an {\em operad\/} (in the monoidal category of
vector spaces) is a 
collection $\calP = \{\calP(n)\}_{n \geq 0}$ of vector
spaces such that
\begin{itemize}
\item[(i)]
there are `circ-$i$' operations $\calP(m) \ot \calP(n) \to
\calP(m+n-1)$, $m,n \geq 0$, $1 \leq i \leq m$, and
\item[(ii)]
each $\calP(n)$ is a right module over the symmetric group $\Sigma_n$,
\item[(iii)] there is the operad unit $1\in \calP(1)$.
\end{itemize}

The above data of course ought to satisfy suitable axioms.
The prominent example is the endomorphism operad $\End_V$ of a vector
space $V$. Its $n$-th component is the space of linear
maps $\otexp Vn \to V$. An {\em algebra\/} over $\calP$ is, by definition,
an operad morphism $\calP \to \End_V$.

\noindent 
{\bf Variants.}
One sometimes does not require the actions in (ii) above. Such operads
are called {\em non-$\Sigma$\/} or {\em non-symmetric operads\/}. 
As may important operads in
this paper are non-$\Sigma$, 
we call an operad with the symmetric group action a {\em $\Sigma$-operad\/}. 
Algebras over non-$\Sigma$ operads
are algebras whose axioms do not involve permutations of
variables. For instance,
associative algebras are algebras over a non-$\Sigma$ operad while
Lie  or commutative associative algebras are~not.

Another type of operads are {\em set-theoretic operads\/},
i.e.~operads that are $\bfk$-linear envelopes of
operads defined in the cartesian monoidal category of sets. Algebras over
set-theoretic operads are algebras whose axioms can be written without the
addition. So associative or associative commutative algebras are
algebras over a set-theoretic operad, Lie or Poisson algebras are~not. 

Sometimes one does not require the existence of the space
$\calP(0)$. Operads of this type describe algebras without
constant operations, such as non-unital associative algebras,
non-unital associative commutative algebras, Lie
algebras, Poisson algebras, \&c.

\noindent 
{\bf Diagonal.}
Given operads $\calP$ and $\calQ$, one can form their (tensor) product
$\calP \otred \calQ$ whose $n$th component, by definition, equals
$\calP(n) \otred \calQ(n)$. A {\em Hopf operad\/} is
equipped with an operad morphism (a {\em diagonal\/}) $\nabla : \calP
\to \calP \otred \calP$. 

A typical Hopf operad is the operad $\UAss$
for unital associative algebras, the operad $\Lie$ for Lie algebras is
not a Hopf operad. It is easy to show that each
set-theoretic operad is a Hopf operad. On the level of algebras,
$\calP$ being Hopf means that the tensor product $A' \otred A''$ of
two $\calP$-algebras is naturally a $\calP$-algebra again.

%\bibliography{b}
\def\cprime{$'$} \def\cprime{$'$}

\end{document}